\documentclass[12pt,fleqn,draft]{article} 
\usepackage{amsmath,amssymb,amsthm,amsfonts,bm}
\usepackage{enumerate}
\usepackage{indentfirst}
\usepackage{cite}
\usepackage[usenames,dvipsnames]{color}

\makeatletter
\def\@cite#1#2{[{{\bfseries #1}\if@tempswa , #2\fi}]}
\renewcommand{\section}{%
\@startsection{section}{1}{\z@}
{0.5truecm plus -1ex minus -.2ex}%
{1.0ex plus .2ex}{\bfseries\large}}
\def\@seccntformat#1{\csname the#1\endcsname.\ }
\makeatother

\numberwithin{equation}{section} 
\theoremstyle{theorem}
\newtheorem{thm}{Theorem}[section]

\newtheorem{lem}[thm]{Lemma}

\theoremstyle{definition}
\newtheorem{df}{Definition}[section]
\newtheorem{remark}{Remark}[section]

\newtheorem*{prth1.1}{Proof of Theorem 1.1}
\newtheorem*{prth1.2}{Proof of Theorem 1.2}
\newtheorem*{prth1.3}{Proof of Theorem 1.3}

\newcommand{\ep}{\varepsilon}

\let\widehat\widehat

\def\Pi{\widehat\pi}

\begin{document}
%==========================title==========================
\begin{center}
    \Large{{\bf 
Nonlocal to local convergence of singular phase field \\ systems 
of conserved type}}
\end{center}
\vspace{5pt}
%===========================author=========================
\begin{center}
    Shunsuke Kurima%
   %\footnote{Corresponding author}
    \\
    \vspace{2pt}
    Department of Mathematics, 
    Tokyo University of Science\\
    1-3, Kagurazaka, Shinjuku-ku, Tokyo 162-8601, Japan\\
    {\tt shunsuke.kurima@gmail.com}\\
    \vspace{2pt}
\end{center}
\begin{center}    
    \small \today
\end{center}

\vspace{2pt}
%=====================  Abstract  =======================
\newenvironment{summary}
{\vspace{.5\baselineskip}\begin{list}{}{%
     \setlength{\baselineskip}{0.85\baselineskip}
     \setlength{\topsep}{0pt}
     \setlength{\leftmargin}{12mm}
     \setlength{\rightmargin}{12mm}
     \setlength{\listparindent}{0mm}
     \setlength{\itemindent}{\listparindent}
     \setlength{\parsep}{0pt}
     \item\relax}}{\end{list}\vspace{.5\baselineskip}}
\begin{summary}
{\footnotesize {\bf Abstract.} 
This paper deals with 
a singular nonlocal phase field system of conserved type.
Colli--K.\ [Nonlinear Anal.\ 190 (2020)] have derived existence of solutions to 
a singular phase field system of conserved type. 
On the other hand, 
Davoli--Scarpa--Trussardi [Arch. Ration. Mech. Anal.\ 239 (2021)] have studied 
nonlocal to local convergence of Cahn-Hilliard equations. 
In this paper 
we prove existence of solutions to 
a nonlocal singular phase field system of conserved type 
whose kernel is not $W^{1, 1}$ 
and  
focus on nonlocal to local convergence of 
singular phase field systems of conserved type.   
}
\end{summary}
\vspace{10pt}

\newpage
%%==============================================================%%
%%==============                                  ==============%%
%%======                      Section1                    ======%%
%%====                                                      ====%%
%%==                                                         ==%%
%%====               Introduction Results            ====%%
%%======                                                  ======%%
%%==============                                  ==============%%
%%==============================================================%%

\section{Introduction}\label{Sec1}

Phase field systems of conserved type have been studied. 
In particular, existence of solutions to 
the singular phase field system of conserved type 
\begin{equation*}\tag{E0}\label{E0}
     \begin{cases}
         (\ln  \theta)_{t} + \varphi_{t} - \Delta\theta = f   
         & \mbox{in}\ \Omega\times(0, T), 
 \\[2.5mm]
         \varphi_{t} - \Delta \mu = 0 
         & \mbox{in}\ \Omega\times(0, T), 
 \\[2.5mm]
         \mu = \tau\varphi_{t} - \Delta \varphi + \xi + \pi(\varphi) - \theta, \ 
         \xi \in \beta(\varphi)     
         & \mbox{in}\ \Omega\times(0, T), 
 \\[2.5mm]
         \partial_{\nu}\theta + \theta = \theta_{\Gamma}, \ 
         \partial_{\nu}\mu = \partial_{\nu}\varphi = 0                                   
         & \mbox{on}\ \Gamma\times(0, T),
 \\[2.5mm]
        (\ln \theta)(0) = \ln \theta_{0},\ \varphi(0)=\varphi_{0}                                          
         & \mbox{in}\ \Omega  
     \end{cases}
\end{equation*}
has been proved by Colli--K. \cite{CK2}, 
where $\tau \geq 0$, 
$\Omega$ is a bounded domain in $\mathbb{R}^d$ ($d = 1, 2, 3$)
with smooth boundary $\Gamma:=\partial\Omega$, 
$T > 0$, 
$\beta \subset  \mathbb{R}\times\mathbb{R}$ is a maximal monotone graph, 
$\pi : \mathbb{R}\to\mathbb{R}$ is an anti-monotone function, 
$\partial_\nu$ denotes differentiation with respect to
the outward normal of $\partial\Omega$, 
and 
$f$, $\theta_{\Gamma}$, $\theta_{0}$, $\varphi_{0}$ are given functions. 
On the other hand, 
in the case that $d = 2, 3$, 
for the two problems 
\begin{equation*}\tag*{(E1)$_{\ep}$}\label{E1ep}
     \begin{cases}
         (\varphi_{\ep})_{t} - \Delta \mu_{\ep} = 0 
         & \mbox{in}\ \Omega\times(0, T), 
 \\[2.5mm]
         \mu_{\ep} = \tau_{\ep}(\varphi_{\ep})_{t} 
         + (K_{\ep}\ast1)\varphi_{\ep} - K_{\ep}\ast\varphi_{\ep} 
 \notag \\ 
         \hspace{50mm} +\ \xi_{\ep} + \pi(\varphi_{\ep}) - g_{\ep}, \ 
         \xi_{\ep} \in \beta(\varphi_{\ep})     
         & \mbox{in}\ \Omega\times(0, T), 
 \\[2.5mm]
         \partial_{\nu}\mu_{\ep} = 0                                   
         & \mbox{on}\ \Gamma\times(0, T),
 \\[2.5mm]
        \varphi_{\ep}(0)=\varphi_{0, \ep}                                          
         & \mbox{in}\ \Omega   
     \end{cases}
\end{equation*}
and 
\begin{equation*}\tag{E1}\label{E1}
     \begin{cases}
         \varphi_{t} - \Delta \mu = 0 
         & \mbox{in}\ \Omega\times(0, T), 
 \\[1.5mm]
         \mu = \tau\varphi_{t} - \Delta\varphi 
                  + \xi + \pi(\varphi) - g, \ \xi \in \beta(\varphi)     
         & \mbox{in}\ \Omega\times(0, T), 
 \\[1.5mm]
         \partial_{\nu}\mu = \partial_{\nu}\varphi = 0                                   
         & \mbox{on}\ \Gamma\times(0, T),
 \\[1.5mm]
        \varphi(0)=\varphi_{0}                                          
         & \mbox{in}\ \Omega,   
     \end{cases}
\end{equation*}
Davoli--Scarpa--Trussardi \cite{DST2021} have confirmed 
that existence of solutions to \eqref{E1} 
can be derived 
by passing to the limit in \ref{E1ep} as $\ep = \ep_{j} \searrow 0$, 
where $\{\ep_{j}\}_{j}$ is a subsequence of $\{\ep\}_{\ep > 0}$. 
Here $\ep > 0$, $\{\tau_{\ep}\}_{\ep > 0} \subset (0, +\infty)$, 
$\tau_{\ep} \to \tau$ as $\ep \searrow 0$, 
$g$, $g_{\ep}$, $\varphi_{0, \ep}$ are given functions, 
and $K_{\ep}$ is a kernel function which is not $W^{1, 1}$. 
However, 
nonlocal singular phase field systems of conserved type whose kernel is not $W^{1, 1}$ 
and 
nonlocal to local convergence of singular phase field systems of conserved type 
seem not be studied yet. 

\bigskip

In this paper we consider the two problems  
%-----------------------------------------------------------------
%
%                                        (Pep)
%
%-----------------------------------------------------------------
 \begin{equation*}\tag*{(P)$_{\ep}$}\label{Pep}
     \begin{cases}
         (\ln  \theta_{\ep})_{t} + (\varphi_{\ep})_{t} - \Delta\theta_{\ep} = f   
         & \mbox{in}\ \Omega\times(0, T), 
 \\[1.5mm]
         (\varphi_{\ep})_{t} - \Delta \mu_{\ep} = 0 
         & \mbox{in}\ \Omega\times(0, T), 
 \\[1.5mm]
         \mu_{\ep} = (\varphi_{\ep})_{t} 
         + (K_{\ep}\ast1)\varphi_{\ep} - K_{\ep}\ast\varphi_{\ep} 
 \notag \\ 
         \hspace{50mm} +\ \xi_{\ep} + \pi(\varphi_{\ep}) - \theta_{\ep}, \ 
         \xi_{\ep} \in \beta(\varphi_{\ep})     
         & \mbox{in}\ \Omega\times(0, T), 
 \\[1.5mm]
         \partial_{\nu}\theta_{\ep} + \theta_{\ep} = \theta_{\Gamma}, \ 
         \partial_{\nu}\mu_{\ep} = 0                                   
         & \mbox{on}\ \Gamma\times(0, T),
 \\[1.5mm]
        (\ln \theta_{\ep})(0) = \ln \theta_{0},\ 
        \varphi_{\ep}(0)=\varphi_{0, \ep}                                          
         & \mbox{in}\ \Omega  
     \end{cases}
\end{equation*}
and 
%-----------------------------------------------------------------
%
%                                        (P)
%
%-----------------------------------------------------------------
 \begin{equation*}\tag{P}\label{P}
     \begin{cases}
         (\ln  \theta)_{t} + \varphi_{t} - \Delta\theta = f   
         & \mbox{in}\ \Omega\times(0, T), 
 \\[1.5mm]
         \varphi_{t} - \Delta \mu = 0 
         & \mbox{in}\ \Omega\times(0, T), 
 \\[1.5mm]
         \mu = \varphi_{t} - \Delta\varphi 
                  + \xi + \pi(\varphi) - \theta, \ \xi \in \beta(\varphi)     
         & \mbox{in}\ \Omega\times(0, T), 
 \\[1.5mm]
         \partial_{\nu}\theta + \theta = \theta_{\Gamma}, \  
         \partial_{\nu}\mu = \partial_{\nu}\varphi = 0                                   
         & \mbox{on}\ \Gamma\times(0, T),
 \\[1.5mm]
        (\ln \theta)(0) = \ln \theta_{0},\ 
        \varphi(0)=\varphi_{0}                                          
         & \mbox{in}\ \Omega,   
     \end{cases}
\end{equation*}
prove existence of solutions to \ref{Pep} 
and verify that existence of solutions to \eqref{P} 
can be established 
by passing to the limit in \ref{Pep}. 
Here $\Omega$ is a bounded domain in $\mathbb{R}^d$ ($d = 2, 3$)
with smooth boundary $\Gamma:=\partial\Omega$. 
Moreover, in this paper we assume that 
\begin{enumerate} 
\setlength{\itemsep}{1mm}
\item[{\bf C1}] 
$\rho_{\ep} : \mathbb{R} \to [0, +\infty)$, 
$\rho_{\ep} \in L_{\mathrm{loc}}^{1}(\mathbb{R})$, 
$\rho_{\ep}(r) = \rho_{\ep}(-r)$ ($\forall r \in \mathbb{R}$, $\forall \ep >0$), 
\[
\int_{0}^{+\infty} \rho_{\ep}(r)r^{d-1}\,dr = c_{d}
\quad (\forall \ep > 0), 
\qquad 
\lim_{\ep \searrow 0} \int_{\delta}^{+\infty} \rho_{\ep}(r)r^{d-1}\,dr = 0 
\quad (\forall \delta > 0),  
\]
where $c_{d} := \dfrac{2}{\int_{S^{d-1}}|e_{1}\cdot\sigma|^2\,d{\cal H}^{d-1}(\sigma)}$. 
The function $K_{\ep} : \Omega\times\Omega \to [0, +\infty)$ is defined as  
$K_{\ep}(x, y) := \dfrac{\rho_{\ep}(|x-y|)}{|x-y|^2}$ for a.a.\ $x, y \in \Omega$.   
\item[{\bf C2}] $\beta \subset \mathbb{R}\times\mathbb{R}$ 
is a maximal monotone graph, with effective domain $D(\beta)$, which coincides with
the subdifferential $\partial\widehat{\beta}$ of a proper lower semicontinuous convex function $\widehat{\beta} : \mathbb{R} \to [0, +\infty]$, 
with effective domain $D(\widehat{\beta}) \supseteq D(\beta) $,  
such that $\widehat{\beta}(0) = 0$.   
\item[{\bf C3}] $\pi : \mathbb{R} \to \mathbb{R}$ is a Lipschitz continuous function. 
\item[{\bf C4}] $f \in L^2(\Omega\times(0, T))$. 
\item[{\bf C5}] $\theta_{\Gamma} \in L^{\infty}(\Gamma\times(0, T))$, 
                $\theta_{0} \in L^{\infty}(\Omega)$   
                and there exist constants $\theta^{*}, \theta_{*} > 0$ such that 
                \[
                \theta_{*} \leq \theta_{\Gamma} \leq \theta^{*} 
                \quad\mbox{a.e.\ on}\ \Gamma\times(0, T) 
                \qquad \mbox{and}\qquad  
                \theta_{*} \leq \theta_0 \leq \theta^{*} 
                \quad\mbox{a.e.\ in}\ \Omega.  
                \]
\item[{\bf C6}] $\varphi_0 \in H^1(\Omega)$ 
                 and $\widehat{\beta}(\varphi_{0}) \in L^1(\Omega)$; 
                 moreover, the mean value 
                 $(\varphi_{0})_{\Omega} := \frac{1}{|\Omega|}\int_{\Omega}\varphi_{0}$ 
                 lies in $\mbox{Int}\,D(\beta)$. 
\item[{\bf C7}] Let $\varphi_{0, \ep} \in L^2(\Omega)$ satisfy that 
there exists a constant $c_{1} > 0$ such that 
\[ 
\|\varphi_{0, \ep}\|_{L^2(\Omega)}^2 \leq c_{1},\  
\int_{\Omega}\int_{\Omega} 
                    K_{\ep}(x, y)|\varphi_{0, \ep}(x) - \varphi_{0, \ep}(y)|^2\,dxdy \leq c_{1},\ 
\|\widehat{\beta}(\varphi_{0, \ep})\|_{L^1(\Omega)} \leq c_{1} 
\]
for all $\ep \in (0, 1)$ 
and there exist constants $a_{0}, b_{0} \in \mathbb{R}$ such that 
$[a_{0}, b_{0}] \subset \mbox{Int}\,D(\beta)$ and 
$a_{0} \leq (\varphi_{0, \ep})_{\Omega} \leq b_{0}$ for all $\ep \in (0, 1)$. 
In addition, 
$\varphi_{0, \ep} \to \varphi_{0}$ weakly in $L^2(\Omega)$ as $\ep \searrow 0$.   
\end{enumerate}
\begin{remark}
The kernel function $K_{\ep}$ in {\bf C1} is not $W^{1, 1}$. 
\end{remark}

\smallskip

To show existence for \ref{Pep} 
we deal with the approximate problem 
%-----------------------------------------------------------------
%
%                                        (Peplambda)
%
%-----------------------------------------------------------------
 \begin{equation*}\tag*{(P)$_{\ep, \lambda}$}\label{Peplam}
     \begin{cases}
         \bigl(\mbox{\rm Ln$_{\lambda}$}(\theta_{\ep, \lambda})\bigr)_{t}  
                                              + (\varphi_{\ep, \lambda})_{t}  
         - \Delta\theta_{\ep, \lambda} = f   
         & \mbox{in}\ \Omega\times(0, T), 
 \\[1.5mm]
         (\varphi_{\ep, \lambda})_{t} - \Delta \mu_{\ep, \lambda} = 0 
         & \mbox{in}\ \Omega\times(0, T), 
 \\[1.5mm]
         \mu_{\ep, \lambda} = (\varphi_{\ep, \lambda})_{t} 
                                      - \lambda\Delta\varphi_{\ep, \lambda} 
         + (K_{\ep}\ast1)\varphi_{\ep, \lambda} - K_{\ep}\ast\varphi_{\ep, \lambda} 
 \notag \\ 
         \hspace{55mm} +\ \beta_{\lambda}(\varphi_{\ep, \lambda}) 
         + \pi(\varphi_{\ep, \lambda}) 
         - \theta_{\ep, \lambda}    
         & \mbox{in}\ \Omega\times(0, T), 
 \\[1.5mm]
         \partial_{\nu}\theta_{\ep, \lambda} + \theta_{\ep, \lambda} = \theta_{\Gamma}, \ 
         \partial_{\nu}\mu_{\ep, \lambda} = \partial_{\nu}\varphi_{\ep, \lambda} = 0                                   
         & \mbox{on}\ \Gamma\times(0, T),
 \\[1.5mm]
        (\mbox{\rm Ln$_{\lambda}$}(\theta_{\ep, \lambda}))(0) 
        = \mbox{\rm Ln$_{\lambda}$}(\theta_0),\ 
        \varphi_{\ep, \lambda}(0)=\varphi_{0, \ep, \lambda}                                          
         & \mbox{in}\ \Omega, 
     \end{cases}
\end{equation*}
where $\lambda > 0$, 
$\mbox{\rm Ln$_{\lambda}$}(r):=\lambda r + \ln_{\lambda}(r)$ 
for $r \in \mathbb{R}$, 
and 
$\ln_{\lambda}$ is the Yosida approximation operator of $\ln$ on $\mathbb{R}$,   
$\beta_{\lambda} : \mathbb{R} \to \mathbb{R}$ is the 
Yosida approximation operator of $\beta$ on $\mathbb{R}$. 
Moreover, in this paper we assume further that  
\begin{enumerate} 
\item[{\bf C8}] Let $\varphi_{0, \ep, \lambda} \in H^1(\Omega)$ 
satisfy that 
$(\varphi_{0, \ep, \lambda})_{\Omega} = (\varphi_{0, \ep})_{\Omega}$ 
for all $\ep \in (0, 1)$ and all $\lambda \in (0, 1)$, and 
there exists a constant $c_{2} > 0$ such that 
\[
\lambda\|\varphi_{0, \ep, \lambda}\|_{H^1(\Omega)}^2 \leq c_{2},\ 
\|\widehat{\beta}(\varphi_{0, \ep, \lambda})\|_{L^1(\Omega)} \leq c_{2}
\]
for all $\ep \in (0, 1)$ and all $\lambda \in (0, 1)$. 
Moreover, $\varphi_{0, \ep, \lambda} \to \varphi_{0, \ep}$ strongly in $L^2(\Omega)$ 
as $\lambda \searrow 0$ for all $\ep \in (0, 1)$. 
\end{enumerate}

\bigskip

Let us define the Hilbert spaces 
   $$
   H:=L^2(\Omega), \quad V:=H^1(\Omega)
   $$
 with inner products 
 \begin{align*} 
 &(u_{1}, u_{2})_{H}:=\int_{\Omega}u_{1}u_{2}\,dx \quad \mbox{for}\ u_{1}, u_{2} \in H, \\
 &(v_{1}, v_{2})_{V}:=
 \int_{\Omega}\nabla v_{1}\cdot\nabla v_{2}\,dx + \int_{\Omega} v_{1}v_{2}\,dx 
 \quad \mbox{for}\ v_{1}, v_{2} \in V,
\end{align*}
 respectively,
 and with the related Hilbertian norms. 
 Moreover, we use the notation 
   $$
   W:=\bigl\{z\in H^2(\Omega)\ |\ \partial_{\nu}z = 0 \quad 
   \mbox{a.e.\ on}\ \partial\Omega\bigr\}.
   $$ 
 The notation $V^{*}$ denotes the dual space of $V$ with 
 duality pairing $\langle\cdot, \cdot\rangle_{V^*, V}$. 
 Moreover, in this paper, the bijective mapping $F : V \to V^{*}$ and 
 the inner product in $V^{*}$ are defined as 
    \begin{align*}
    &\langle Fv_{1}, v_{2} \rangle_{V^*, V} := 
    (v_{1}, v_{2})_{V} \quad \mbox{for}\ v_{1}, v_{2}\in V,   
    %\label{defF}
    \\[1mm]
    &(v_{1}^{*}, v_{2}^{*})_{V^{*}} := 
    \left\langle v_{1}^{*}, F^{-1}v_{2}^{*} 
    \right\rangle_{V^*, V} 
    \quad \mbox{for}\ v_{1}^{*}, v_{2}^{*}\in V^{*}.   
    %\label{innerVstar}
    \end{align*}
This article employs the Hilbert space 
  $$
  V_{0}:=\left\{ z \in H^1(\Omega)\ \Big{|} \ \int_{\Omega} z = 0 \right\}   
  $$
 with inner product 
 \begin{align*} 
 (v_{1}, v_{2})_{V_{0}}:=
 \int_{\Omega}\nabla v_{1}\cdot\nabla v_{2}\,dx 
 \quad \mbox{for}\ v_{1}, v_{2} \in V_{0} 
\end{align*}
 and with the related Hilbertian norm.
 The notation $V_{0}^{*}$ denotes the dual space of $V_{0}$ with 
 duality pairing $\langle\cdot, \cdot\rangle_{V_{0}^*, V_{0}}$. 
 Moreover, in this paper, the bijective mapping ${\cal N} : V_{0}^{*} \to V_{0}$ and 
 the inner product in $V_{0}^{*}$ are specified by  
    \begin{align*}
    &\langle v^{*}, v \rangle_{V_{0}^*, V_{0}} := 
    \int_{\Omega} \nabla {\cal N}v^{*} \cdot \nabla v 
    \quad \mbox{for}\ v^{*} \in V_{0}^*, v \in V_{0},   
    %\label{defN}
    \\[3mm]
    &(v_{1}^{*}, v_{2}^{*})_{V_{0}^{*}} := 
    \left\langle v_{1}^{*}, {\cal N}v_{2}^{*} 
    \right\rangle_{V_{0}^*, V_{0}} 
    \quad \mbox{for}\ v_{1}^{*}, v_{2}^{*}\in V_{0}^{*}.  
    %\label{innerVzerostar}
    \end{align*}
For $\ep > 0$ we define 
\[
V_{\ep} := \left\{ \varphi \in H  \ \Bigg{|} \ 
                              \int_{\Omega}\int_{\Omega} 
                                 K_{\ep}(x, y)|\varphi(x) - \varphi(y)|^2\,dxdy < + \infty \right\}
\]
and 
\[
E_{\ep}(\varphi) := \frac{1}{4}\int_{\Omega}\int_{\Omega} 
                                              K_{\ep}(x, y)|\varphi(x) - \varphi(y)|^2\,dxdy   
\quad \mbox{for}\ \varphi \in V_{\ep}. 
\]
We set $a_{\ep} : V_{\ep} \times V_{\ep} \to \mathbb{R}$ as 
\[
a_{\ep}(\varphi, \psi) := 
\frac{1}{2}\int_{\Omega}\int_{\Omega} 
                 K_{\ep}(x, y)(\varphi(x) - \varphi(y))(\psi(x) - \psi(y))\,dxdy  
\quad \mbox{for}\ \varphi, \psi \in V_{\ep}
\]
and define 
\begin{align*}
&W_{\ep} := 
\{ \varphi \in V_{\ep}  \ | \ 
                     \exists g \in H \ \mbox{s.t.}\ 
                     a_{\ep}(\varphi, \psi) = (g, \psi)_{H}\ \mbox{for all}\ \psi \in V_{\ep} \}, 
\\[5mm]
&B_{\ep}(\varphi)(x) := \int_{\Omega} K_{\ep}(x, y)(\varphi(x) - \varphi(y))\,dy 
\quad \mbox{for a.a.}\ x \in \Omega \ \mbox{and all}\ \varphi \in W_{\ep}. 
\end{align*}
Moreover, we set 
\begin{align*}
&\|\varphi\|_{V_{\ep}} 
  := \sqrt{\|\varphi\|_{H}^2 + 2E_{\ep}(\varphi)} 
    \quad \mbox{for}\ \varphi \in V_{\ep}, 
\\[3mm]
&\|\varphi\|_{W_{\ep}} 
  := \sqrt{\|\varphi\|_{H}^2 + \|B_{\ep}(\varphi)\|_{H}^2} 
   \quad \mbox{for}\ \varphi \in W_{\ep}, 
\\[3mm]
&(\varphi_{1}, \varphi_{2})_{V_{\ep}} 
  := (\varphi_{1}, \varphi_{2})_{H} + a_{\ep}(\varphi_{1}, \varphi_{2}) 
\quad \mbox{for}\ \varphi_{1}, \varphi_{2} \in V_{\ep}, 
\\[3mm]
&(\varphi_{1}, \varphi_{2})_{W_{\ep}} 
  := (\varphi_{1}, \varphi_{2})_{H} + (B_{\ep}(\varphi_{1}), B_{\ep}(\varphi_{2}))_{H}
\quad \mbox{for}\ \varphi_{1}, \varphi_{2} \in W_{\ep}.  
\end{align*}

\bigskip

We define weak solutions of \ref{Peplam}, \ref{Pep} and \eqref{P} as follows. 
%-----------------------------------------------------------------
%
%                              Definition of  solutions to (Peplam)
%
%-----------------------------------------------------------------
%
 \begin{df}
 A triplet $(\theta_{\ep, \lambda}, \mu_{\ep, \lambda}, \varphi_{\ep, \lambda})$ with 
    \begin{align*}
    &\theta_{\ep, \lambda} \in L^2(0, T; V),\ 
      \mbox{\rm Ln$_{\lambda}$}(\theta_{\ep, \lambda}) 
                                                    \in H^1(0, T; V^{*}) \cap L^{\infty}(0, T; H), \\ 
   &\mu_{\ep, \lambda} \in L^2(0, T; W), \\ 
   &\varphi_{\ep, \lambda} 
      \in H^1(0, T; H) \cap L^{\infty}(0, T; V) \cap L^2(0, T; W) 
    \end{align*}
 is called a {\it weak solution} of \ref{Peplam} 
 if $(\theta_{\ep, \lambda}, \mu_{\ep, \lambda}, \varphi_{\ep, \lambda})$ satisfies 
    \begin{align}
        &\langle \bigl(\mbox{\rm Ln$_{\lambda}$}(\theta_{\ep, \lambda})\bigr)_{t}, 
                                                                                       w \rangle_{V^{*}, V}
           + ((\varphi_{\ep, \lambda})_{t}, w)_{H} 
           + \int_{\Omega} \nabla \theta_{\ep, \lambda} \cdot \nabla w 
      \notag \\
      &\hspace{35mm} + \int_{\Gamma} (\theta_{\ep, \lambda} - \theta_{\Gamma})w 
      = (f, w)_{H} 
                 \quad  \mbox{a.e.\ in}\ (0, T) 
                           \ \  \mbox{for all}\ w \in V,  \label{dfsoleplam1}
     \\[3.5mm]
        &(\varphi_{\ep, \lambda})_{t} - \Delta \mu_{\ep, \lambda} = 0 
                 \quad  \mbox{a.e.\ in}\ \Omega\times(0, T),  \label{dfsoleplam2}
    \\[2mm]
        &\mu_{\ep, \lambda} = (\varphi_{\ep, \lambda})_{t} 
                                        - \lambda \Delta \varphi_{\ep, \lambda} 
                                        + B_{\ep}(\varphi_{\ep, \lambda}) 
                                        + \beta_{\lambda}(\varphi_{\ep, \lambda}) 
                                        + \pi(\varphi_{\ep, \lambda}) - \theta_{\ep, \lambda} 
          \quad \mbox{a.e.\ in}\ \Omega\times(0, T), \label{dfsoleplam3}
     \\[2mm]
     &\partial_{\nu}\mu_{\ep, \lambda} = \partial_{\nu}\varphi_{\ep, \lambda} = 0 
         \quad \mbox{a.e.\ on}\ \Gamma\times(0, T),    \label{dfsoleplam4}
     \\[2mm]
        &(\mbox{\rm Ln$_{\lambda}$}(\theta_{\ep, \lambda}))(0) 
                                                   = \mbox{\rm Ln$_{\lambda}$}(\theta_{0}),\ 
          \varphi_{\ep, \lambda}(0) = \varphi_{0, \ep, \lambda} 
                                                \quad \mbox{a.e.\ in}\ \Omega.  \label{dfsoleplam5}
     \end{align}
 \end{df}
%-----------------------------------------------------------------
%
%                              Definition of  solutions to (Pep)
%
%-----------------------------------------------------------------
%
 \begin{df}         
 A quadruple $(\theta_{\ep}, \mu_{\ep}, \varphi_{\ep}, \xi_{\ep})$ with 
    \begin{align*}
    &\theta_{\ep} \in L^2(0, T; V),\ 
      \ln\theta_{\ep} \in H^1(0, T; V^{*}) \cap L^{\infty}(0, T; H), \\ 
   &\mu_{\ep} \in L^2(0, T; W), \\ 
   &\varphi_{\ep} \in H^1(0, T; H) \cap L^{\infty}(0, T; V_{\ep}) \cap L^2(0, T; W_{\ep}), \\
   &\xi_{\ep} \in L^2(0, T; H)
    \end{align*}
 is called a {\it weak solution} of \ref{Pep} 
 if $(\theta_{\ep}, \mu_{\ep}, \varphi_{\ep}, \xi_{\ep})$ satisfies 
    \begin{align*}
        &\langle (\ln\theta_{\ep})_{t}, w \rangle_{V^{*}, V}
           + ((\varphi_{\ep})_{t}, w)_{H} + \int_{\Omega} \nabla \theta_{\ep} \cdot \nabla w 
      \\
      &\hspace{47mm} + \int_{\Gamma} (\theta_{\ep} - \theta_{\Gamma})w = (f, w)_{H} 
                 \quad  \mbox{a.e.\ in}\ (0, T) 
                           \ \  \mbox{for all}\ w \in V,  %\label{dfsolep1}
     \\[3.5mm]
        &(\varphi_{\ep})_{t} - \Delta \mu_{\ep} = 0 
           \quad \mbox{a.e.\ in}\ \Omega\times(0, T),  %\label{dfsolep2}
    \\[2mm]
        &\mu_{\ep} = (\varphi_{\ep})_{t} + B_{\ep}(\varphi_{\ep}) 
                                                   + \xi_{\ep} + \pi(\varphi_{\ep}) - \theta_{\ep},\ 
                                                                         \xi_{\ep} \in \beta(\varphi_{\ep}) 
                               \quad \mbox{a.e.\ in}\ \Omega\times(0, T), %\label{dfsolep3}
     \\[2mm]
        &(\ln\theta_{\ep})(0) = \ln\theta_0,\ \varphi_{\ep}(0) = \varphi_{0, \ep}
                                                \quad \mbox{a.e.\ in}\ \Omega. %\label{dfsolep4}
     \end{align*}
 \end{df}
%-----------------------------------------------------------------
%
%                              Definition of  solutions to (P)
%
%-----------------------------------------------------------------
%
 \begin{df}         
 A quadruple $(\theta, \mu, \varphi, \xi)$ with 
    \begin{align*}
    &\theta \in L^2(0, T; V),\ 
      \ln\theta \in H^1(0, T; V^{*}) \cap L^{\infty}(0, T; H), \\ 
    &\mu \in L^2(0, T; W), \\ 
    &\varphi \in H^1(0, T; H) \cap L^{\infty}(0, T; V) \cap L^2(0, T; W), \\
    &\xi \in L^2(0, T; H)
    \end{align*}
 is called a {\it weak solution} of \eqref{P} 
 if $(\theta, \mu, \varphi, \xi)$ satisfies 
    \begin{align*}
        &\langle (\ln\theta)_{t}, w \rangle_{V^{*}, V}
           + (\varphi_{t}, w)_{H} + \int_{\Omega} \nabla \theta \cdot \nabla w 
      \\
      &\hspace{50mm} + \int_{\Gamma} (\theta - \theta_{\Gamma})w = (f, w)_{H} 
                 \quad  \mbox{a.e.\ in}\ (0, T) 
                           \ \  \mbox{for all}\ w \in V,  %\label{dfsol1}
     \\[3.5mm]
        &\varphi_{t} - \Delta \mu = 0 
           \quad \mbox{a.e.\ in}\ \Omega\times(0, T), %\label{dfsol2}
    \\[2mm]
        &\mu = \varphi_{t} - \Delta\varphi + \xi + \pi(\varphi) - \theta,\ 
                                                                                    \xi \in \beta(\varphi) 
                               \quad \mbox{a.e.\ in}\ \Omega\times(0, T), %\label{dfsol3}
     \\[2mm]
        &(\ln\theta)(0) = \ln\theta_0,\ \varphi(0) = \varphi_0
                                                       \quad \mbox{a.e.\ in}\ \Omega. %\label{dfsol4}
     \end{align*}
 \end{df}

\bigskip

Now the main results read as follows. 
\begin{thm}\label{maintheorem1}
Assume that {\rm {\bf C1}-{\bf C8}} hold. 
Then 
for all $\ep \in (0, 1)$ and all $\lambda \in (0, 1)$ 
there exists a unique weak solution 
$(\theta_{\ep, \lambda}, \mu_{\ep, \lambda}, \varphi_{\ep, \lambda})$ 
of {\rm \ref{Peplam}}. 
\end{thm}
\begin{thm}\label{maintheorem2}
Assume that {\rm {\bf C1}-{\bf C8}} hold. 
Then 
for all $\ep \in (0, 1)$ 
there exists a weak solution 
$(\theta_{\ep}, \mu_{\ep}, \varphi_{\ep}, \xi_{\ep})$ of {\rm \ref{Pep}}. 
Moreover, 
the solution components $\theta_{\ep}$ and $\varphi_{\ep}$ are unique, 
and the solution components $\mu_{\ep}$ and $\xi_{\ep}$ 
are unique if $\beta$ is single-valued. 
\end{thm}
\begin{thm}\label{maintheorem3}
Assume that {\rm {\bf C1}-{\bf C8}} hold. 
Then existence of weak solutions to \eqref{P} 
can be proved by passing to the limit in {\rm \ref{Pep}}  
as $\ep = \ep_{j} \searrow 0$, 
where $\{\ep_{j}\}_{j}$ is a subsequence of $\{\ep\}_{\ep > 0}$. 
\end{thm}

\bigskip

This paper is organized as follows. 
In Section \ref{Sec2} we give useful results for proving the main theorems. 
Sections \ref{Sec3}, \ref{Sec4} and \ref{Sec5} 
are devoted to the proofs of 
Theorems \ref{maintheorem1}, \ref{maintheorem2} and \ref{maintheorem3}, 
respectively.

\vspace{10pt}

%%==============================================================%%
%%==============                                  ==============%%
%%======                      Section2                    ======%%
%%====                                                      ====%%
%%==                                                          ==%%
%%====                                                       ====%%
%%======                                                  ======%%
%%==============                                  ==============%%
%%==============================================================%%
\section{Preliminaries}\label{Sec2}

In this section we will provide some results 
which will be used later for the proofs 
of Theorems \ref{maintheorem1}, \ref{maintheorem2} and \ref{maintheorem3}. 
\begin{lem}[{\cite[Lemma 1]{DST2021}}]\label{PL1}
For all $\ep >0$ the following properties hold: 
\begin{enumerate}
\setlength{\itemsep}{3mm}
\item[(i)] The bilinear forms $(\cdot, \cdot)_{V_{\ep}}$ and $(\cdot, \cdot)_{W_{\ep}}$ 
are scalar products on $V_{\ep}$ and $W_{\ep}$ 
inducing the norms $\|\cdot\|_{V_{\ep}}$ and $\|\cdot\|_{W_{\ep}}$, 
respectively. 
Moreover, $V_{\ep}$ and $W_{\ep}$ are Hilbert spaces.  
\item[(ii)] For all $\sigma \in (0, 1]$ 
we have $C^{0, \sigma}(\overline{\Omega}) \hookrightarrow W_{\ep}$ 
continuously and there exists a constant $C_{\ep, \sigma} > 0$ satisfying 
\[
B_{\ep}(\varphi) \in L^{\infty}(\Omega),  
\quad 
\|B_{\ep}(\varphi)\|_{L^{\infty}(\Omega)} 
\leq C_{\ep, \sigma}\|\varphi\|_{C^{0, \sigma}(\overline{\Omega})} 
\qquad \mbox{for all}\ \varphi \in C^{0, \sigma}(\overline{\Omega}). 
\]
\item[(iii)] The inclusions $W_{\ep} \hookrightarrow V_{\ep} \hookrightarrow H$ 
are continuous and dense. 
\item[(iv)] The unbounded linear operator 
$B_{\ep} : D(B_{\ep}) = W_{\ep} \subset H \to H$ 
is maximal monotone. 
\item[(v)] The operator $B_{\ep} : D(B_{\ep}) = W_{\ep} \subset H \to H$ 
extends to a bounded linear operator $B_{\ep} : V_{\ep} \to {V_{\ep}}^{*}$ 
fulfilling 
\[
\|B_{\ep}(\varphi)\|_{{V_{\ep}}^{*}} \leq \|\varphi\|_{V_{\ep}}  
\quad \mbox{for all}\ \varphi \in V_{\ep}. 
\]
Moreover, such extension coincides with the linear operator 
$A_{\ep} : V_{\ep} \to {V_{\ep}}^{*}$ 
associated to the bilinear form $a_{\ep}$, defined as 
\[
A_{\ep}(\varphi) := a_{\ep}(\varphi, \cdot) \quad \mbox{for}\ \varphi \in {V_{\ep}}. 
\]
\item[(vi)] The map $E_{\ep} : V_{\ep} \to [0, +\infty)$ is convex, lower 
semicontinuous and is of class $C^{1}$. 
Moreover, $DE_{\ep} = B_{\ep} : V_{\ep} \to {V_{\ep}}^{*}$ 
in the sense of G\^{a}teaux. 
\end{enumerate}
\end{lem}

\begin{lem}[{\cite[Lemma 2]{DST2021}}]\label{PL2}
The inclusion $V \hookrightarrow V_{\ep}$ is continuous 
and there exists a constant $C > 0$, independent of $\ep$, such that 
\[
\|\varphi\|_{V_{\ep}} \leq C\|\varphi\|_{V} 
\quad \mbox{for all}\ \varphi \in V. 
\] 
It holds that 
\begin{align*}
&\lim_{\ep \searrow 0} E_{\ep}(\varphi) 
= \frac{1}{2}\int_{\Omega} |\nabla\varphi|^2 
\quad \mbox{for all}\ \varphi \in V, 
\\[3mm]
&\lim_{\ep \searrow 0} 
     \langle B_{\ep}(\varphi_{1}), \varphi_{2} \rangle_{{V_{\ep}}^{*}, V_{\ep}}
= \int_{\Omega} \nabla \varphi_{1} \cdot \nabla \varphi_{2} 
\quad \mbox{for all}\ \varphi_{1}, \varphi_{2} \in V. 
\end{align*}
Moreover, 
for all $\varphi \in H$ and all sequence $\{\varphi_{\ep}\}_{\ep > 0} \subset H$ 
with $\varphi_{\ep} \to \varphi$ in $H$ as $\ep \searrow 0$, 
the inequality 
\[
\liminf_{\ep \searrow 0} E_{\ep}(\varphi_{\ep}) 
\geq E(\varphi)
\]
holds, where 
the map $E : H \to [0, +\infty]$ is defined as 
\[
E(\varphi) := 
\begin{cases}
\frac{1}{2}\int_{\Omega} |\nabla\varphi|^2  & \varphi \in V, 
\\ 
+ \infty  & \varphi \in H \setminus V. 
\end{cases}
\]
\end{lem}

\begin{lem}[{\cite[Lemma 3]{DST2021}}]\label{PL3}
For all $\delta > 0$ there exist constants $C_{\delta} > 0$ 
and $\ep_{\delta} > 0$ such that 
for all sequence $\{\varphi_{\ep}\}_{\ep \in (0, \ep_{\delta})} \subset V_{\ep}$ 
the inequality 
\[
\|\varphi_{\ep_{1}} - \varphi_{\ep_{2}}\|_{H}^2 
\leq \delta\bigl(E_{\ep_{1}}(\varphi_{\ep_{1}}) + E_{\ep_{2}}(\varphi_{\ep_{2}})\bigr) 
      + C_{\delta}\|\varphi_{\ep_{1}} - \varphi_{\ep_{2}}\|_{V^{*}}^2
\]
holds for all $\ep_{1}, \ep_{2} \in (0, \ep_{\delta})$. 
\end{lem}

\begin{lem}[{\cite[Section 4.1]{DST2021}}]\label{PL4}
For all $\theta \in L^2(0, T; H)$, $\ep \in (0, 1)$ and $\lambda \in (0, 1)$ 
there exists a unique solution $(\overline{\varphi}, \overline{\mu}) 
\in H^1(0, T; H) \times L^2(0, T; W)$ 
of the problem 
\begin{align*}
&\overline{\varphi}_{t} - \Delta \overline{\mu} = 0 
                 \quad  \mbox{a.e.\ in}\ \Omega\times(0, T),  
    \\[2mm]
        &\overline{\mu} = \overline{\varphi}_{t} 
                                        - \lambda \Delta \overline{\varphi}
                                        + B_{\ep}(\overline{\varphi}) 
                                        + \beta_{\lambda}(\overline{\varphi}) 
                                        + \pi(\overline{\varphi}) - \theta  
          \quad \mbox{a.e.\ in}\ \Omega\times(0, T), 
     \\[2mm]
     &\partial_{\nu}\overline{\mu} = \partial_{\nu}\overline{\varphi} = 0 
         \quad \mbox{a.e.\ on}\ \Gamma\times(0, T),    
     \\[2mm]
        &\overline{\varphi}(0) = \varphi_{0, \ep, \lambda} 
                                                \quad \mbox{a.e.\ in}\ \Omega. 
\end{align*}
\end{lem}

\begin{lem}[{cf.\ \cite[Proof of Lemma 3.4]{CL1998}}]\label{PL5}
For all $\varphi \in H^1(0, T; H)$, $\ep \in (0, 1)$ and $\lambda \in (0, 1)$ 
there exists a unique solution $\overline{\theta} \in L^2(0, T; H)$ 
of the problem 
\begin{align*}
&\left\langle (\lambda\overline{\theta} + \ln_{\lambda}(\overline{\theta}))_{t}, w 
                                                                                        \right\rangle_{V^{*}, V}
           + \int_{\Omega} \nabla \overline{\theta} \cdot \nabla w 
           + \int_{\Gamma} \overline{\theta}w 
\\ 
&\hspace{35mm} = (f, w)_{H} + \int_{\Gamma} \theta_{\Gamma}w 
   - (\varphi_{t}, w)_{H} 
            \quad  \mbox{a.e.\ in}\ (0, T) \ \  \mbox{for all}\ w \in V, 
\\[2mm]
&\mbox{\rm Ln$_{\lambda}$}(\overline{\theta})(0) 
                                           = \mbox{\rm Ln$_{\lambda}$}(\theta_{0}) 
                                                               \quad \mbox{a.e.\ in}\ \Omega. 
\end{align*}
\end{lem}

\begin{lem}[{\cite[Proof of Theorem 3 with $a=b=0$]{O-1983}}]\label{PL6}
 Let $\beta \subset \mathbb{R} \times \mathbb{R}$ 
 be a multi-valued maximal monotone function.     
 Then
    \begin{align*}
    \bigl(-\Delta u, \beta_{\lambda}(u)\bigr)_{H} \geq 0 
    \quad \mbox{for all}\ 
    u\in W\ \mbox{and all}\ \lambda > 0,  
    \end{align*}
 where 
 $\beta_{\lambda}$ is the Yosida approximation of $\beta$ on $\mathbb{R}$. 
\end{lem}

\vspace{10pt}

%%==============================================================%%
%%==============                                  ==============%%
%%======                      Section3                    ======%%
%%====                                                      ====%%
%%==                                                          ==%%
%%====                                                       ====%%
%%======                                                  ======%%
%%==============                                  ==============%%
%%==============================================================%%
\section{Proof of Theorem \ref{maintheorem1}}\label{Sec3}

In this section we will prove Theorem \ref{maintheorem1}.
\begin{prth1.1}
Thanks to Lemmas \ref{PL4} and \ref{PL5}, 
we can define ${\cal A}$, ${\cal B}$, ${\cal S}$ as 
\begin{align*}
&{\cal A}\theta = \overline{\varphi} \quad \mbox{for}\ \theta \in L^2(0, T; H),  
\\[2mm]
&{\cal B}\varphi = \overline{\theta} 
\quad \mbox{for}\ \varphi \in H^1(0, T; H),   
\\[2mm]
&X := \left\{ w \in L^2(0, T; H) \ \Bigg{|} \  
                 \int_{0}^{T} e^{-Ls}\|w(s)\|_{H}^2\,ds < + \infty 
         \right\}, 
\\[2mm]
&d(z, w) :=  \int_{0}^{T} e^{-Ls}\|z(s) - w(s)\|_{H}^2\,ds  \quad \mbox{for}\ z, w \in X, 
\\[2mm]
&{\cal S} := {\cal B}\circ{\cal A} : X \to X,  
\end{align*}
where $L > 0$ is some constant to be fixed later 
and $\overline{\varphi}$, $\overline{\theta}$ are as in Lemmas \ref{PL4} and \ref{PL5}, 
respectively. 
Now we let $\theta_{1}, \theta_{2} \in X$. 
Then we have from Lemma \ref{PL5} that 
\begin{align}\label{tomato1}
&\lambda({\cal S}\theta_{1} - {\cal S}\theta_{2}, w)_{H} 
+ (\ln_{\lambda}({\cal S}\theta_{1}) - \ln_{\lambda}({\cal S}\theta_{2}), w)_{H}  
+ ({\cal A}\theta_{1} - {\cal A}\theta_{2}, w)_{H}  
\notag \\ 
&+ \int_{\Omega} \nabla\left(\int_{0}^{t} ({\cal S}\theta_{1} - {\cal S}\theta_{2})\,ds \right) 
                                                                                                     \cdot \nabla w 
\notag \\
&+ \int_{\Gamma} \left(\int_{0}^{t} ({\cal S}\theta_{1} - {\cal S}\theta_{2})\,ds \right)w 
= 0 
\quad  \mbox{in}\ (0, T) \ \  \mbox{for all}\ w \in V. 
\end{align}
Taking $w = e^{-Lt}({\cal S}\theta_{1} - {\cal S}\theta_{2})$ in \eqref{tomato1}, 
integrating over $(0, T)$ 
and integrating by parts with respect to time 
lead to the identity 
\begin{align*}
&\lambda\int_{0}^{T} e^{-Ls}\|{\cal S}\theta_{1} - {\cal S}\theta_{2}\|_{H}^2\,ds 
+ \int_{0}^{T} e^{-Ls}
    (\ln_{\lambda}({\cal S}\theta_{1}) - \ln_{\lambda}({\cal S}\theta_{2}), 
                                                       {\cal S}\theta_{1} - {\cal S}\theta_{2})_{H}\,ds 
\\
&+ \int_{0}^{T} e^{-Ls}({\cal A}\theta_{1} - {\cal A}\theta_{2}, 
                                               {\cal S}\theta_{1} - {\cal S}\theta_{2})_{H}\,ds 
\\
&+ \frac{1}{2}e^{-LT}\left( 
     \left\|\nabla\int_{0}^{T} ({\cal S}\theta_{1} - {\cal S}\theta_{2})\,ds \right\|_{H}^2 
     + \int_{\Gamma} \left|\int_{0}^{T} ({\cal S}\theta_{1} - {\cal S}\theta_{2})\,ds \right|^2  
     \right) 
\\ 
& + \frac{L}{2}\int_{0}^{T} e^{-Ls}
      \left(
      \left\|\nabla\int_{0}^{s}({\cal S}\theta_{1} - {\cal S}\theta_{2})\,dr \right\|_{H}^2 
      + \int_{\Gamma}\left|\int_{0}^{s} ({\cal S}\theta_{1} - {\cal S}\theta_{2})\,dr \right|^2
      \right)\,ds 
= 0,
\end{align*}
and then we see from the monotonicity of $\ln_{\lambda}$ and the Young inequality that  
\begin{equation}\label{tomato2}
\lambda d({\cal S}\theta_{1}, {\cal S}\theta_{2}) 
\leq \frac{1}{\lambda} d({\cal A}\theta_{1}, {\cal A}\theta_{2}). 
\end{equation}
Here we put 
\begin{align*}
\overline{\mu}_{j} := ({\cal A}\theta_{j})_{t} - \lambda\Delta{\cal A}\theta_{j} 
             + B_{\ep}({\cal A}\theta_{j}) + \beta_{\lambda}({\cal A}\theta_{j}) 
             + \pi({\cal A}\theta_{j}) - \theta_{j}. 
\end{align*}
It follows from Lemma \ref{PL4} that 
\begin{align}
&({\cal A}\theta_{1} - {\cal A}\theta_{2})_{t} 
   - \Delta(\overline{\mu}_{1} - \overline{\mu}_{2}) = 0 
\quad  \mbox{a.e.\ in}\ \Omega\times(0, T),  \label{tomato3}
\\[1mm] 
&\partial_{\nu}(\overline{\mu}_{1} - \overline{\mu}_{2}) = 0 
\quad  \mbox{a.e.\ on}\ \Gamma\times(0, T),   \label{tomato4}
\\[1mm]  
&({\cal A}\theta_{1} - {\cal A}\theta_{2})(0) = 0 
\quad  \mbox{a.e.\ in}\ \Omega\times(0, T),    \label{tomato5}
\\[3mm] 
&\overline{\mu}_{1} - \overline{\mu}_{2} 
= ({\cal A}\theta_{1} - {\cal A}\theta_{2})_{t} 
   - \lambda\Delta({\cal A}\theta_{1} - {\cal A}\theta_{2}) 
    + B_{\ep}({\cal A}\theta_{1} - {\cal A}\theta_{2}) 
\notag \\ 
&\hspace{20mm}
    + \beta_{\lambda}({\cal A}\theta_{1}) - \beta_{\lambda}({\cal A}\theta_{2}) 
\notag \\ 
&\hspace{20mm}
    + \pi({\cal A}\theta_{1}) - \pi({\cal A}\theta_{2}) 
    - (\theta_{1} - \theta_{2}) 
\quad  \mbox{a.e.\ in}\ \Omega\times(0, T),  \label{tomato6}
\\[1mm] 
&\partial_{\nu}({\cal A}\theta_{1} - {\cal A}\theta_{2}) = 0 
\quad  \mbox{a.e.\ on}\ \Gamma\times(0, T).  \label{tomato7}
\end{align}
Since $({\cal A}\theta_{1} - {\cal A}\theta_{2})_{\Omega} = 0$ 
by \eqref{tomato3}--\eqref{tomato5}, 
we deduce from \eqref{tomato3} that 
\begin{align}\label{tomato8}
&\frac{1}{2}\frac{d}{dt}\|{\cal A}\theta_{1} - {\cal A}\theta_{2}\|_{V_{0}^*}^2 
+ (\overline{\mu}_{1} - \overline{\mu}_{2}, {\cal A}\theta_{1} - {\cal A}\theta_{2})_{H} 
\notag \\ 
&= \langle 
        ({\cal A}\theta_{1} - {\cal A}\theta_{2})_{t}, 
                            {\cal N}({\cal A}\theta_{1} - {\cal A}\theta_{2})
     \rangle_{V_{0}^*, V_{0}} 
    + (\overline{\mu}_{1} - \overline{\mu}_{2} 
                                - (\overline{\mu}_{1} - \overline{\mu}_{2})_{\Omega}, 
                                                         {\cal A}\theta_{1} - {\cal A}\theta_{2})_{H} 
\notag \\ 
&= \langle 
        ({\cal A}\theta_{1} - {\cal A}\theta_{2})_{t}, 
                            {\cal N}({\cal A}\theta_{1} - {\cal A}\theta_{2})
     \rangle_{V_{0}^*, V_{0}} 
\notag \\ 
  &\,\quad 
   + \langle 
         -\Delta (\overline{\mu}_{1} - \overline{\mu}_{2} 
                                - (\overline{\mu}_{1} - \overline{\mu}_{2})_{\Omega}), 
                                                   {\cal N}({\cal A}\theta_{1} - {\cal A}\theta_{2})
     \rangle_{V_{0}^*, V_{0}} 
\notag \\ 
&= 0.
\end{align}
We derive from \eqref{tomato6} and \eqref{tomato7} that 
\begin{align}\label{tomato9}
&(\overline{\mu}_{1} - \overline{\mu}_{2}, {\cal A}\theta_{1} - {\cal A}\theta_{2})_{H} 
\notag \\[2mm]
&= \frac{1}{2}\frac{d}{dt}\|{\cal A}\theta_{1} - {\cal A}\theta_{2}\|_{H}^2 
   + \lambda\|\nabla({\cal A}\theta_{1} - {\cal A}\theta_{2})\|_{H}^2 
   + (B_{\ep}({\cal A}\theta_{1} - {\cal A}\theta_{2}), 
                                                   {\cal A}\theta_{1} - {\cal A}\theta_{2})_{H} 
\notag \\ 
 &\,\quad 
       + (\beta_{\lambda}({\cal A}\theta_{1}) - \beta_{\lambda}({\cal A}\theta_{2}), 
                                                              {\cal A}\theta_{1} - {\cal A}\theta_{2})_{H} 
   + (\pi({\cal A}\theta_{1}) - \pi({\cal A}\theta_{2}), 
                                                          {\cal A}\theta_{1} - {\cal A}\theta_{2})_{H} 
\notag \\ 
   &\,\quad - (\theta_{1} - \theta_{2}, {\cal A}\theta_{1} - {\cal A}\theta_{2})_{H}. 
\end{align}
We multiply \eqref{tomato8}, \eqref{tomato9} by $e^{-Lt}$, 
integrate over $(0, T)$, 
integrate by parts with respect to time,  
use the monotonicity of $\beta_{\lambda}$, {\bf C3}, 
the Young inequality 
to infer that 
\begin{equation*}
\left(\frac{L}{2} - \|\pi'\|_{L^{\infty}(\mathbb{R})} - \frac{1}{2\lambda} \right)
d({\cal A}\theta_{1}, {\cal A}\theta_{2}) 
\leq \frac{\lambda}{2} d(\theta_{1}, \theta_{2}). 
\end{equation*}
Thus, fixing $L$ to $2\|\pi'\|_{L^{\infty}(\mathbb{R})} + \frac{3}{\lambda}$, we obtain 
\begin{equation}\label{tomato10}
\frac{1}{\lambda} d({\cal A}\theta_{1}, {\cal A}\theta_{2}) 
\leq \frac{\lambda}{2} d(\theta_{1}, \theta_{2}).  
\end{equation}
Therefore combining \eqref{tomato2} and \eqref{tomato10} implies that  
\begin{equation*}
d({\cal S}\theta_{1}, {\cal S}\theta_{2}) 
\leq \frac{1}{2} d(\theta_{1}, \theta_{2}),
\end{equation*}
whence ${\cal S} : X \to X$ is a contraction mapping in $X$ 
and then by the Banach fixed-point theorem 
there exists a unique function $\theta \in X$ such that 
$\theta = {\cal S}\theta$. 
Thus, putting $\varphi := {\cal A}\theta$, 
we can prove existence of weak solutions to \ref{Peplam}.

Next we show uniqueness of weak solutions to \ref{Peplam}. 
We let $(\theta_{j}, \mu_{j}, \varphi_{j})$, $j = 1, 2$, be two weak solutions of \ref{Peplam}. 
We have from \eqref{dfsoleplam2} that 
\begin{align}\label{tomato11}
&\frac{1}{2}\frac{d}{dt}\|\varphi_{1}(t) - \varphi_{2}(t)\|_{{V_{0}}^{*}}^2 
+ (\mu_{1}(t) - \mu_{2}(t), \varphi_{1}(t) - \varphi_{2}(t))_{H}
\notag \\ 
&= \left\langle 
         (\varphi_{1})_{t}(t) - (\varphi_{2})_{t}(t), {\cal N}(\varphi_{1}(t) - \varphi_{2}(t))  
    \right\rangle_{{V_{0}}^{*}, V_{0}} 
\notag \\ 
    &\,\quad + (\mu_{1}(t) - \mu_{2}(t) - (\mu_{1}(t) - \mu_{2}(t))_{\Omega}, 
                                                                    \varphi_{1}(t) - \varphi_{2}(t))_{H} 
\notag \\ 
&= \left\langle 
         (\varphi_{1})_{t}(t) - (\varphi_{2})_{t}(t), {\cal N}(\varphi_{1}(t) - \varphi_{2}(t))  
    \right\rangle_{{V_{0}}^{*}, V_{0}} 
\notag \\ 
    &\,\quad + \left\langle 
                       - \Delta (\mu_{1}(t) - \mu_{2}(t) - (\mu_{1}(t) - \mu_{2}(t))_{\Omega}), 
                                                                      {\cal N}(\varphi_{1}(t) - \varphi_{2}(t))  
                   \right\rangle_{{V_{0}}^{*}, V_{0}} 
\notag \\ 
&= 0. 
\end{align}
It follows from \eqref{dfsoleplam3} and \eqref{dfsoleplam4} that 
\begin{align}\label{tomato12}
&(\mu_{1}(t) - \mu_{2}(t), \varphi_{1}(t) - \varphi_{2}(t))_{H} 
\notag \\ 
&= \frac{1}{2}\frac{d}{dt}\|\varphi_{1}(t) - \varphi_{2}(t)\|_{H}^2 
     + \lambda\|\nabla(\varphi_{1}(t) - \varphi_{2}(t))\|_{H}^2 
     + 2E_{\ep}(\varphi_{1}(t) - \varphi_{2}(t)) 
\notag \\
  &\,\quad + (\beta_{\lambda}(\varphi_{1}(t)) - \beta_{\lambda}(\varphi_{2}(t)), 
                                                                     \varphi_{1}(t) - \varphi_{2}(t))_{H} 
\notag \\ 
 &\,\quad + (\pi(\varphi_{1}(t)) - \pi(\varphi_{2}(t)), \varphi_{1}(t) - \varphi_{2}(t))_{H} 
\notag \\ 
 &\,\quad - (\theta_{1}(t) - \theta_{2}(t), \varphi_{1}(t) - \varphi_{2}(t))_{H}. 
\end{align}
We deduce from \eqref{dfsoleplam1} and \eqref{dfsoleplam5} that 
\begin{align}\label{tomato13}
&(\mbox{\rm Ln$_{\lambda}$}(\theta_{1}(t)) - \mbox{\rm Ln$_{\lambda}$}(\theta_{2}(t)), 
                                                                               \theta_{1}(t) - \theta_{2}(t))_{H} 
+ (\theta_{1}(t) - \theta_{2}(t), \varphi_{1}(t) - \varphi_{2}(t))_{H} 
\notag \\ 
&+ \frac{1}{2}\frac{d}{dt}
      \left(
         \left\|\nabla\left(\int_{0}^{t}(\theta_{1}(s) - \theta_{2}(s))\,ds\right) \right\|_{H}^2 
         + \left\|\int_{0}^{t}(\theta_{1}(s) - \theta_{2}(s))\,ds \right\|_{L^2(\Gamma)}^2 
      \right) 
\notag \\ 
&= 0.
\end{align}
Therefore we can verify uniqueness of weak solutions to \ref{Peplam} 
by \eqref{tomato11}--\eqref{tomato13}, integrating over $(0, t)$, where $t \in [0, T]$, 
the monotonicity of $\beta_{\lambda}$ and $\mbox{\rm Ln$_{\lambda}$}$, 
{\bf C3}, and the Gronwall lemma. 
\qed
\end{prth1.1}

\vspace{10pt}

%%==============================================================%%
%%==============                                  ==============%%
%%======                      Section4                    ======%%
%%====                                                      ====%%
%%==                                                          ==%%
%%====                                                       ====%%
%%======                                                  ======%%
%%==============                                  ==============%%
%%==============================================================%%
\section{Proof of Theorem \ref{maintheorem2}}\label{Sec4}

This section will prove Theorem \ref{maintheorem2}. 
We will establish a priori estimates for \ref{Peplam} 
to derive existence for \ref{Pep} 
by passing to the limit in \ref{Peplam}.  
\begin{lem}\label{estieplam1}
There exists a constant $C > 0$ such that 
\begin{align*}
&\|\nabla\mu_{\ep, \lambda}\|_{L^2(0, T; H)} 
+ \|\varphi_{\ep, \lambda}\|_{L^{\infty}(0, T; V_{\ep})} 
+ \|\varphi_{\ep, \lambda}\|_{H^1(0, T; H)} 
+ \lambda^{1/2}\|\nabla\varphi_{\ep, \lambda}\|_{L^{\infty}(0, T; H)} 
\\
&+ \|\widehat{\beta}_{\lambda}(\varphi_{\ep, \lambda})\|_{L^{\infty}(0, T; L^1(\Omega))}
\leq C, 
\\[3mm] 
&\lambda^{1/2}\|\theta_{\ep, \lambda}\|_{L^{\infty}(0, T; H)} 
+ \lambda^{1/2}\|\ln_{\lambda}(\theta_{\ep, \lambda})\|_{L^{\infty}(0, T; H)} 
+ \|J_{\lambda}^{\ln}(\theta_{\ep, \lambda})\|_{L^{\infty}(0, T; L^1(\Omega))} 
\\ 
&+ \|\theta_{\ep, \lambda}\|_{L^2(0, T; V)} \leq C 
\end{align*}
for all $\ep \in (0, 1)$ and all $\lambda \in (0, 1)$, 
where $\widehat{\beta}_{\lambda}(r) := \int_{0}^{r} \beta_{\lambda}(s)\,ds$ for $r \in \mathbb{R}$ 
and $J_{\lambda}^{\ln}$ is the resolvent of $\ln$ on $\mathbb{R}$. 
\end{lem}
\begin{proof}
We see from \eqref{dfsoleplam2} that 
\begin{align}\label{a1}
&\int_{0}^{t} \|\nabla\mu_{\ep, \lambda}(s)\|_{H}^2\,ds 
+ \int_{0}^{t} ((\varphi_{\ep, \lambda})_{t}(s), \mu_{\ep, \lambda}(s))_{H}\,ds 
\notag \\ 
&= \int_{0}^{t} ((\varphi_{\ep, \lambda})_{t}(s) - \Delta \mu_{\ep, \lambda}(s), 
                                                                    \mu_{\ep, \lambda}(s))_{H}\,ds 
= 0. 
\end{align}
Also, thanks to \eqref{dfsoleplam3}, we can obtain  
\begin{align}\label{a2}
&\int_{0}^{t} ((\varphi_{\ep, \lambda})_{t}(s), \mu_{\ep, \lambda}(s))_{H}\,ds 
\notag \\ 
&= \int_{0}^{t} \|(\varphi_{\ep, \lambda})_{t}(s)\|_{H}^2\,ds 
     + \frac{\lambda}{2}\|\nabla\varphi_{\ep, \lambda}(t)\|_{H}^2 
     + E_{\ep}(\varphi_{\ep, \lambda}(t)) 
     + \int_{\Omega} (\widehat{\beta}_{\lambda} + \widehat{\pi})(\varphi_{\ep, \lambda}(t)) 
\notag \\ 
&= \frac{\lambda}{2}\|\nabla\varphi_{0, \ep, \lambda}\|_{H}^2 
     + E_{\ep}(\varphi_{0, \ep, \lambda}) 
     + \int_{\Omega} (\widehat{\beta}_{\lambda} + \widehat{\pi})(\varphi_{0, \ep, \lambda}) 
\notag \\ 
&\,\quad - \int_{0}^{t} (\theta_{\ep, \lambda}(s), (\varphi_{\ep, \lambda})_{t}(s))_{H}\,ds,   
\end{align}
where $\widehat{\pi}(r) := \int_{0}^{r} \pi(s)\,ds$ for $r \in \mathbb{R}$. 
On the other hand, noting that 
\[
\theta_{\ep, \lambda} 
= \lambda\ln_{\lambda}(\theta_{\ep, \lambda}) + J_{\lambda}^{\ln}(\theta_{\ep, \lambda}),\ 
\ln_{\lambda}(\theta_{\ep, \lambda}) = \ln(J_{\lambda}^{\ln}(\theta_{\ep, \lambda})), 
\]
taking $w = \theta_{\ep, \lambda}(t)$ in \eqref{dfsoleplam1}, 
and integrating over $(0, t)$, 
we infer that 
\begin{align}\label{a3}
&\frac{\lambda}{2}\|\theta_{\ep, \lambda}(t)\|_{H}^2 
+ \frac{\lambda}{2}\|\ln_{\lambda}(\theta_{\ep, \lambda}(t))\|_{H}^2 
+ \int_{\Omega} J_{\lambda}^{\ln}(\theta_{\ep, \lambda}(t)) 
\notag \\  
&+ \int_{0}^{t} (\theta_{\ep, \lambda}(s), (\varphi_{\ep, \lambda})_{t}(s))_{H}\,ds 
\notag \\ 
&+ \int_{0}^{t} \|\nabla\theta_{\ep, \lambda}(s)\|_{H}^2\,ds 
+ \int_{0}^{t} \|\theta_{\ep, \lambda}(s)\|_{L^2(\Gamma)}^2\,ds 
\notag \\[2mm] 
&= \frac{\lambda}{2}\|\theta_{0}\|_{H}^2 
+ \frac{\lambda}{2}\|\ln_{\lambda}(\theta_{0})\|_{H}^2 
+ \int_{\Omega} J_{\lambda}^{\ln}(\theta_{0}) 
\notag \\  
&\,\quad 
 + \int_{0}^{t} \left(\int_{\Gamma}\theta_{\ep, \lambda}(s)\theta_{\Gamma}(s)\right)\,ds 
    + \int_{0}^{t} (f(s), \theta_{\ep, \lambda}(s))_{H}\,ds. 
\end{align}
Therefore, since 
\[
|\ln_{\lambda}(\theta_{0})| \leq |\ln \theta_{0}| \leq \max_{\theta_{*} \leq r \leq \theta^{*}} |\ln r|
\]
by {\bf C5} and there exist constants $C^{*}, C_{*} > 0$ such that 
\begin{equation*}%\label{keyineq}
C_{*}\bigl(\|\nabla w\|_{H}^2 + \|w\|_{L^2(\Gamma)}^2 \bigr) 
\leq \|w\|_{V}^2 
\leq C^{*}\bigl(\|\nabla w\|_{H}^2 + \|w\|_{L^2(\Gamma)}^2 \bigr) 
\end{equation*}
for all $w \in V$ (see, e.g., \cite[p.\ 20]{N1967}), 
we can prove Lemma \ref{estieplam1} 
by \eqref{a1}--\eqref{a3}, the Young inequality, {\bf C3}, {\bf C7}, {\bf C8}, 
the identity $J_{\lambda}^{\ln}(\theta_{0}) = - \lambda\ln_{\lambda}(\theta_{0}) + \theta_{0}$, 
and the inequalities $0 \leq \widehat{\beta_{\lambda}}(r) \leq \widehat{\beta}(r)$ 
(see e.g., \cite[Theorem 2.9, p.\ 48]{Barbu2}). 
\end{proof}

\begin{lem}\label{estieplam2}
There exists a constant $C > 0$ such that 
\[
\|\bigl(\mbox{\rm Ln$_{\lambda}$}(\theta_{\ep, \lambda})\bigr)_{t}\|_{L^2(0, T; V^{*})}
\leq C 
\]
for all $\ep \in (0, 1)$ and all $\lambda \in (0, 1)$. 
\end{lem}
\begin{proof}
This lemma holds by \eqref{dfsoleplam1} and Lemma \ref{estieplam1}. 
\end{proof}

\begin{lem}\label{estieplam3}
There exists a constant $C > 0$ such that 
\[
\|\mu_{\ep, \lambda}\|_{L^2(0, T; W)} \leq C 
\]
for all $\ep \in (0, 1)$ and all $\lambda \in (0, 1)$. 
\end{lem}
\begin{proof}
Since 
$(\varphi_{\ep, \lambda}(t))_{\Omega} = (\varphi_{0, \ep, \lambda})_{\Omega} 
= (\varphi_{0, \ep})_{\Omega}$ 
by \eqref{dfsoleplam2}, \eqref{dfsoleplam4} and {\bf C8}, 
it holds that 
\begin{align}\label{c1}
&\frac{1}{2}\frac{d}{dt}
    \|\varphi_{\ep, \lambda}(t) - (\varphi_{0, \ep})_{\Omega}\|_{{V_{0}}^{*}}^2 
+ (\mu_{\ep, \lambda}(t), \varphi_{\ep, \lambda}(t) - (\varphi_{0, \ep})_{\Omega})_{H}
\notag \\ 
&= \left\langle (\varphi_{\ep, \lambda})_{t}(t), 
                       {\cal N}(\varphi_{\ep, \lambda}(t) - (\varphi_{0, \ep})_{\Omega}) 
    \right\rangle_{{V_{0}}^{*}, V_{0}}
+ (\mu_{\ep, \lambda}(t) - (\mu_{\ep, \lambda}(t))_{\Omega}, 
                                  \varphi_{\ep, \lambda}(t) - (\varphi_{0, \ep})_{\Omega})_{H} 
\notag \\ 
&= \left\langle (\varphi_{\ep, \lambda})_{t}(t), 
                       {\cal N}(\varphi_{\ep, \lambda}(t) - (\varphi_{0, \ep})_{\Omega}) 
    \right\rangle_{{V_{0}}^{*}, V_{0}} 
\notag \\ 
  &\,\quad 
        + \left\langle - \Delta (\mu_{\ep, \lambda}(t) - (\mu_{\ep, \lambda}(t))_{\Omega}), 
                                     {\cal N}(\varphi_{\ep, \lambda}(t) - (\varphi_{0, \ep})_{\Omega}) 
    \right\rangle_{{V_{0}}^{*}, V_{0}} 
\notag \\ 
&= 0. 
\end{align}
Also, since $\int_{\Omega} B_{\ep}(\varphi_{\ep, \lambda}(t)) = 0$, 
it follows from \eqref{dfsoleplam3} that 
\begin{align}\label{c2}
&(\mu_{\ep, \lambda}(t), \varphi_{\ep, \lambda}(t) - (\varphi_{0, \ep})_{\Omega})_{H}
\notag \\ 
&= \frac{1}{2}\frac{d}{dt}\|\varphi_{\ep, \lambda}(t) - (\varphi_{0, \ep})_{\Omega}\|_{H}^2 
     + \lambda\|\nabla\varphi_{\ep, \lambda}(t)\|_{H}^2 
     + (B_{\ep}(\varphi_{\ep, \lambda}(t)), \varphi_{\ep, \lambda}(t))_{H} 
\notag \\ 
&\,\quad + (\beta_{\lambda}(\varphi_{\ep, \lambda}(t)), 
                       \varphi_{\ep, \lambda}(t) - (\varphi_{0, \ep})_{\Omega})_{H} 
     + (\pi(\varphi_{\ep, \lambda}(t)) - \theta_{\ep, \lambda}(t), 
                       \varphi_{\ep, \lambda}(t) - (\varphi_{0, \ep})_{\Omega})_{H}.  
\end{align}
Here, owing to {\bf C7}, there exist constants $C_{1}, C_{2} > 0$ such that 
\begin{equation}\label{c3}
\|\beta_{\lambda}(\varphi_{\ep, \lambda}(t))\|_{L^1(\Omega)} 
\leq C_{1}(\beta_{\lambda}(\varphi_{\ep, \lambda}(t)), 
                       \varphi_{\ep, \lambda}(t) - (\varphi_{0, \ep})_{\Omega})_{H} + C_{2}
\end{equation}
for a.a.\ $t \in (0, T)$ and all $\ep \in (0, 1)$, $\lambda \in (0, 1)$ 
(see e.g., \cite[Section 5, p.\ 908]{GMS-2009}). 
Thus by \eqref{c1}--\eqref{c3}, 
integrating over $(0, T)$, {\bf C3}, {\bf C7}, Lemma \ref{estieplam1}, 
there exists a constant $C_{3} > 0$ such that 
\begin{equation*}
\|\beta_{\lambda}(\varphi_{\ep, \lambda})\|_{L^{\infty}(0, T; L^1(\Omega))} \leq C_{3} 
\end{equation*}
for all $\ep \in (0, 1)$ and all $\lambda \in (0, 1)$,  
and then 
\begin{equation}\label{c4}
\|(\mu_{\ep, \lambda})_{\Omega}\|_{L^2(0, T)} 
= \|((\varphi_{\ep, \lambda})_{t})_{\Omega} 
     + (\beta_{\lambda}(\varphi_{\ep, \lambda}))_{\Omega} 
     + (\pi(\varphi_{\ep, \lambda}))_{\Omega} 
     - (\theta_{\ep, \lambda})_{\Omega}
  \|_{L^2(0, T)} 
\leq C_{4} 
\end{equation}
for all $\ep \in (0, 1)$ and all $\lambda \in (0, 1)$ with some $C_{4} > 0$. 
Thus we see from Lemma \ref{estieplam1}, \eqref{c4} 
and the Poincar\'e--Wirtinger inequality that 
there exists a constant $C_{5} > 0$ such that 
\begin{equation}\label{c5}
\|\mu_{\ep, \lambda}\|_{L^2(0, T; H)} \leq C_{5} 
\end{equation}
for all $\ep \in (0, 1)$ and all $\lambda \in (0, 1)$. 
On the other hand, 
we derive from \eqref{dfsoleplam2} and Lemma \ref{estieplam1} that 
there exists a constant $C_{6} > 0$ such that 
\begin{equation}\label{c6}
\|- \Delta \mu_{\ep, \lambda}\|_{L^2(0, T; H)} 
= \|(\varphi_{\ep, \lambda})_{t}\|_{L^2(0, T; H)} 
\leq C_{6} 
\end{equation}
for all $\ep \in (0, 1)$ and all $\lambda \in (0, 1)$. 
Therefore we can obtain Lemma \ref{estieplam3} 
by \eqref{c5}, \eqref{c6} and the elliptic regularity theory. 
\end{proof}

\begin{lem}\label{estieplam4}
There exists a constant $C > 0$ such that 
\[
\|-\lambda\Delta\varphi_{\ep, \lambda} + B_{\ep}(\varphi_{\ep, \lambda})\|_{L^2(0, T; H)} 
+ \|\beta_{\lambda}(\varphi_{\ep, \lambda})\|_{L^2(0, T; H)} 
\leq C 
\]
for all $\ep \in (0, 1)$ and all $\lambda \in (0, 1)$. 
\end{lem}
\begin{proof}
We have from \eqref{dfsoleplam3} that 
\begin{align}\label{d1}
&\|\beta_{\lambda}(\varphi_{\ep, \lambda}(t))\|_{H}^2 
+ (\beta_{\lambda}(\varphi_{\ep, \lambda}(t)), 
          - \lambda\Delta\varphi_{\ep, \lambda}(t) + B_{\ep}(\varphi_{\ep, \lambda}(t)))_{H}
\notag \\ 
&= (\beta_{\lambda}(\varphi_{\ep, \lambda}(t)), 
       \mu_{\ep, \lambda}(t) - (\varphi_{\ep, \lambda})_{t}(t) 
                                               - \pi(\varphi_{\ep, \lambda}(t)) + \theta_{\ep, \lambda}(t))_{H}. 
\end{align}
Moreover, Lemma \ref{PL6} and the monotonicity of $\beta_{\lambda}$ yield that 
\begin{align}\label{d2}
&(\beta_{\lambda}(\varphi_{\ep, \lambda}(t)), 
          - \lambda\Delta\varphi_{\ep, \lambda}(t) + B_{\ep}(\varphi_{\ep, \lambda}(t)))_{H} 
\notag \\ 
&= \lambda(\beta_{\lambda}(\varphi_{\ep, \lambda}(t)), 
                                        - \Delta\varphi_{\ep, \lambda}(t))_{H} 
\notag \\ 
&\,\quad 
 + \frac{1}{2}\int_{\Omega}\int_{\Omega}
                    K_{\ep}(x, y)(\beta_{\lambda}(\varphi_{\ep, \lambda}(x, t)) 
                                              - \beta_{\lambda}(\varphi_{\ep, \lambda}(y, t)))
                                 (\varphi_{\ep, \lambda}(x, t) - \varphi_{\ep, \lambda}(y, t))\,dxdy 
\notag \\ 
&\geq 0. 
\end{align}
Thus we infer from \eqref{d1}, \eqref{d2}, the integration over $(0, T)$,  
the Young inequality, Lemmas \ref{estieplam1}, \ref{estieplam3} 
and {\bf C3} 
that there exists a constant $C_{1} > 0$ such that 
\begin{equation}\label{d3}
\|\beta_{\lambda}(\varphi_{\ep, \lambda})\|_{L^2(0, T; H)} \leq C_{1}
\end{equation}
for all $\ep \in (0, 1)$ and all $\lambda \in (0, 1)$. 
On the other hand, 
it follows from \eqref{dfsoleplam3}, Lemmas \ref{estieplam1} and \ref{estieplam3}, 
and {\bf C3} that 
there exists a constant $C_{2} > 0$ such that 
\begin{equation}\label{d4}
\|- \lambda\Delta\varphi_{\ep, \lambda} + B_{\ep}(\varphi_{\ep, \lambda}) 
                                   + \beta_{\lambda}(\varphi_{\ep, \lambda})\|_{L^2(0, T; H)}
\leq C_{2}
\end{equation}
for all $\ep \in (0, 1)$ and all $\lambda \in (0, 1)$. 
Therefore combining \eqref{d3} and \eqref{d4} leads to Lemma \ref{estieplam4}. 
\end{proof}

\begin{lem}\label{estieplam5}
There exists a constant $C > 0$ such that 
\[
\|\mbox{\rm Ln$_{\lambda}$}(\theta_{\ep, \lambda})\|_{L^{\infty}(0, T; H)}
\leq C 
\]
for all $\ep \in (0, 1)$ and all $\lambda \in (0, 1)$. 
\end{lem}
\begin{proof}
Taking $w = \mbox{\rm Ln$_{\lambda}$}(\theta_{\ep, \lambda}(t))$ in \eqref{dfsoleplam1} 
implies that 
\begin{align}\label{e1}
&\frac{1}{2}\frac{d}{dt}\|\mbox{\rm Ln$_{\lambda}$}(\theta_{\ep, \lambda}(t))\|_{H}^2 
+ \int_{\Omega} \nabla\theta_{\ep, \lambda} \cdot 
                                    \nabla\mbox{\rm Ln$_{\lambda}$}(\theta_{\ep, \lambda}(t)) 
\notag \\ 
&+ \int_{\Gamma} 
             (\theta_{\ep, \lambda}(t) - \theta_{\Gamma}(t))
                 (\mbox{\rm Ln$_{\lambda}$}(\theta_{\ep, \lambda}(t)) 
                                                 - \mbox{\rm Ln$_{\lambda}$}(\theta_{\Gamma}(t))) 
\notag \\ 
&= (f(t) - (\varphi_{\ep, \lambda})_{t}(t), 
                                       \mbox{\rm Ln$_{\lambda}$}(\theta_{\ep, \lambda}(t)))_{H}
     - \int_{\Gamma} (\theta_{\ep, \lambda}(t) - \theta_{\Gamma}(t))
                                                   \mbox{\rm Ln$_{\lambda}$}(\theta_{\Gamma}(t)).  
\end{align}
Here the identity  
\begin{align}\label{e2}
&\int_{\Omega} \nabla\theta_{\ep, \lambda} \cdot 
                                    \nabla\mbox{\rm Ln$_{\lambda}$}(\theta_{\ep, \lambda}(t)) 
\notag \\
&= \lambda\|\nabla\theta_{\ep, \lambda}(t)\|_{H}^2 
   + \lambda\|\nabla\ln_{\lambda}(\theta_{\ep, \lambda}(t))\|_{H}^2 
   + \int_{\Omega} \frac{|\nabla J_{\lambda}^{\ln}(\theta_{\ep, \lambda}(t))|^2}
                                                            {J_{\lambda}^{\ln}(\theta_{\ep, \lambda}(t))}  
\end{align}
holds. 
Thus, since 
\begin{align*}
&|\mbox{\rm Ln$_{\lambda}$}(\theta_{\Gamma})| 
\leq \lambda|\theta_{\Gamma}| + |\ln_{\lambda}(\theta_{\Gamma})| 
\leq \lambda|\theta_{\Gamma}| + |\ln(\theta_{\Gamma})| 
\leq \theta^{*} + \max_{\theta_{*} \leq r \leq \theta^{*}}|\ln r|, 
\\ 
&|\mbox{\rm Ln$_{\lambda}$}(\theta_{0})| 
\leq \lambda|\theta_{0}| + |\ln_{\lambda}(\theta_{0})| 
\leq \lambda|\theta_{0}| + |\ln(\theta_{0})| 
\leq \theta^{*} + \max_{\theta_{*} \leq r \leq \theta^{*}}|\ln r|
\end{align*}
for all $\lambda \in (0, 1)$ by {\bf C5}, 
We can prove Lemma \ref{estieplam5} 
by \eqref{e1}, \eqref{e2}, the monotonicity of \mbox{\rm Ln$_{\lambda}$}, 
integrating over $(0, t)$, the Young inequality, Lemma \ref{estieplam1}, 
the Gronwall lemma. 
\end{proof}

\bigskip

\begin{prth1.2}
Let $\ep \in (0, 1)$. 
Then by Lemmas \ref{estieplam1}-\ref{estieplam5}, 
the Aubin--Lions lemma for the compact embedding $H \hookrightarrow V^*$ 
that there exist some functions 
$u_{\ep}$, $\theta_{\ep}$, $\mu_{\ep}$, $\varphi_{\ep}$, $\xi_{\ep}$, $\eta_{\ep}$ 
such that 
\begin{align*}
    &\theta_{\ep} \in L^2(0, T; V),\ 
      u_{\ep} \in H^1(0, T; V^{*}) \cap L^{\infty}(0, T; H), \\ 
    &\mu_{\ep} \in L^2(0, T; W), \\ 
    &\varphi_{\ep} \in H^1(0, T; H) \cap L^{\infty}(0, T; V_{\ep}), \\
    &\eta_{\ep}, \xi_{\ep} \in L^2(0, T; H)
    \end{align*}
and 
\begin{align}
&\mbox{\rm Ln$_{\lambda}$}(\theta_{\ep, \lambda}) \to u_{\ep} 
\quad \mbox{weakly$^*$ in}\ H^1(0, T; V^*) \cap L^{\infty}(0, T; H),  
\label{weaklam1} 
\\[1.5mm]
&\mbox{\rm Ln$_{\lambda}$}(\theta_{\ep, \lambda}) \to u_{\ep} 
\quad \mbox{strongly in}\ C([0, T]; V^*), 
\label{stronglam1}
\\[1.5mm]
&\theta_{\ep, \lambda} \to \theta_{\ep} 
\quad \mbox{weakly in}\ L^2(0, T; V), 
\label{weaklam2}
\\[1.5mm]
&\varphi_{\ep, \lambda} \to \varphi_{\ep} 
\quad \mbox{weakly$^*$ in}\ H^1(0, T; H) \cap L^{\infty}(0, T; V_{\ep}), 
\label{weaklam3} 
\\[1.5mm]
&\varphi_{\ep, \lambda} \to \varphi_{\ep} 
\quad \mbox{strongly in}\ C([0, T]; V^*), 
\label{stronglam2} 
\\[1.5mm] 
&\lambda\varphi_{\ep, \lambda} \to 0 
\quad \mbox{strongly in}\ L^{\infty}(0, T; V), 
\label{stronglam3} 
\\[1.5mm]
&\mu_{\ep, \lambda} \to \mu_{\ep} 
\quad \mbox{weakly in}\ L^2(0, T; W), 
\label{weaklam4} 
\\[1.5mm]
&\beta_{\lambda}(\varphi_{\ep, \lambda}) \to \xi_{\ep} 
\quad \mbox{weakly in}\ L^2(0, T; H), 
\label{weaklam5}
\\[1.5mm]
&\pi(\varphi_{\ep, \lambda}) \to \zeta_{\ep} 
\quad \mbox{weakly in}\ L^2(0, T; H), 
\label{weaklam6}
\\[1.5mm]
&- \lambda \Delta \varphi_{\ep, \lambda} + B_{\ep}(\varphi_{\ep, \lambda}) \to \eta_{\ep} 
\quad \mbox{weakly in}\ L^2(0, T; H)
\label{weaklam7}
\end{align}
as $\lambda = \lambda_{j} \searrow 0$. 
We derive from \eqref{weaklam3} and Lemma \ref{PL1} (v) that 
\begin{equation*}
B_{\ep}(\varphi_{\ep, \lambda}) \to B_{\ep}(\varphi_{\ep}) 
\quad \mbox{weakly$^*$ in}\ L^{\infty}(0, T; {V_{\ep}}^*)
\end{equation*}
as $\lambda = \lambda_{j} \searrow 0$ 
and then we see from \eqref{stronglam3} and \eqref{weaklam7} that 
\begin{equation*}
B_{\ep}(\varphi_{\ep, \lambda}) \to \eta_{\ep} 
\quad \mbox{weakly$^*$ in}\ L^2(0, T; V^*)
\end{equation*}
as $\lambda = \lambda_{j} \searrow 0$. 
Thus we have 
\begin{equation}\label{pasta1}
B_{\ep}(\varphi_{\ep}) = \eta_{\ep} \in L^2(0, T; H), 
\quad 
\varphi_{\ep} \in L^2(0, T; W_{\ep}). 
\end{equation}
It follows from \eqref{stronglam1} and \eqref{weaklam2} that 
\begin{align*}
&\int_{0}^{T} (\ln_{\lambda}(\theta_{\ep, \lambda}(s)), \theta_{\ep, \lambda}(s))_{H}\,ds 
\notag \\
&= \int_{0}^{T} \left\langle 
                    \mbox{\rm Ln$_{\lambda}$}(\theta_{\ep, \lambda}), \theta_{\ep, \lambda}(s) 
                    \right\rangle_{V^{*}, V}\,ds 
    - \lambda\int_{0}^{T} \|\theta_{\ep, \lambda}(s)\|_{H}^2\,ds
\notag \\ 
&\to \int_{0}^{T} \langle u_{\ep}(s), \theta_{\ep}(s) \rangle_{V^{*}, V}\,ds 
= \int_{0}^{T} (u_{\ep}(s), \theta_{\ep}(s))_{H}\,ds 
\end{align*}
as $\lambda = \lambda_{j} \searrow 0$, 
whence it holds that 
\begin{equation}\label{pasta2}
u_{\ep} = \ln \theta_{\ep}  \quad \mbox{a.e.\ in}\ \Omega\times(0, T)
\end{equation}
(see, e.g., \cite[Lemma 1.3, p.\ 42]{Barbu1}). 
Noting that 
\begin{align*}
&\int_{0}^{T}\int_{\Omega} 
     e^{-4\|\pi'\|_{L^{\infty}(\mathbb{R})}s}
             (\mu_{\ep, \lambda}(x, s) + \theta_{\ep, \lambda}(x, s))
                                                             \varphi_{\ep, \lambda}(x, s)\,dxds  
\notag \\ 
&= \int_{0}^{T} \left\langle 
                       \varphi_{\ep, \lambda}(s), 
                               e^{-4\|\pi'\|_{L^{\infty}(\mathbb{R})}s}
                                         (\mu_{\ep, \lambda}(s) + \theta_{\ep, \lambda}(s)) 
                   \right\rangle_{V^{*}, V}\,ds 
\notag \\ 
&\to \int_{0}^{T} \left\langle 
                         \varphi_{\ep}(s), 
                           e^{-4\|\pi'\|_{L^{\infty}(\mathbb{R})}s}(\mu_{\ep}(s) + \theta_{\ep}(s)) 
                   \right\rangle_{V^{*}, V}\,ds 
\notag \\ 
&= \int_{0}^{T}\int_{\Omega} 
                   e^{-4\|\pi'\|_{L^{\infty}(\mathbb{R})}s}
                              (\mu_{\ep}(x, s) + \theta_{\ep}(x, s))\varphi_{\ep}(x, s)\,dxds 
\end{align*}
as $\lambda = \lambda_{j} \searrow 0$ 
by \eqref{weaklam2}, \eqref{stronglam2} and \eqref{weaklam4}, 
we can show 
\begin{equation}\label{pasta3}
\xi_{\ep} + \zeta_{\ep} + 2\|\pi'\|_{L^{\infty}(\mathbb{R})}\varphi_{\ep} 
\in (\beta + \pi + 2\|\pi'\|_{L^{\infty}(\mathbb{R})}\mbox{id})(\varphi_{\ep}) 
\quad \mbox{a.e.\ in}\ \Omega\times(0, T)
\end{equation}
in reference to \cite[Section 4.3]{DST2021}. 
Also, since 
\begin{align*}
&\int_{0}^{T}\int_{\Omega} 
     e^{-4\|\pi'\|_{L^{\infty}(\mathbb{R})}s}
             (\theta_{\ep, \lambda}(x, s) - \theta_{\ep}(x, s))
                             (\varphi_{\ep, \lambda}(x, s) - \varphi_{\ep}(x, s))\,dxds  
\notag \\ 
&= \int_{0}^{T} \left\langle 
                       \varphi_{\ep, \lambda}(s) - \varphi_{\ep}(s), 
                                             e^{-4\|\pi'\|_{L^{\infty}(\mathbb{R})}s}
                                                        (\theta_{\ep, \lambda}(s) - \theta_{\ep}(s)) 
                   \right\rangle_{V^{*}, V}\,ds 
\notag \\ 
&\to 0
\end{align*}
as $\lambda = \lambda_{j} \searrow 0$, 
we can verify
\begin{equation}\label{pasta4}
\varphi_{\ep, \lambda} \to \varphi_{\ep} 
\quad \mbox{strongly in}\ C([0, T]; H) \cap L^2(0, T; V_{\ep})  
\end{equation}
as $\lambda = \lambda_{j} \searrow 0$ 
in reference to \cite[Section 4.3]{DST2021}. 
Thus we derive from \eqref{weaklam6}, \eqref{pasta4} and {\bf C3} that 
\begin{equation}\label{pasta5}
\zeta_{\ep} = \pi(\varphi_{\ep}) 
\quad \mbox{a.e.\ in}\ \Omega\times(0, T). 
\end{equation}
Moreover, combining \eqref{pasta3} and \eqref{pasta5} implies that 
\begin{equation}\label{pasta6}
\xi_{\ep} \in \beta(\varphi_{\ep}) 
\quad \mbox{a.e.\ in}\ \Omega\times(0, T). 
\end{equation}
Therefore we can prove existence of weak solutions to \ref{Pep} 
by \eqref{weaklam1}--\eqref{pasta2}, \eqref{pasta5}, \eqref{pasta6}, {\bf C8}. 
Moreover, we can confirm the second half of Theorem \ref{maintheorem2} 
in the same way as in the proof of uniqueness for \ref{Peplam}.
\qed
\end{prth1.2}

\vspace{10pt}

%%==============================================================%%
%%==============                                  ==============%%
%%======                      Section5                    ======%%
%%====                                                      ====%%
%%==                                                          ==%%
%%====                                                       ====%%
%%======                                                  ======%%
%%==============                                  ==============%%
%%==============================================================%%
\section{Proof of Theorem \ref{maintheorem3}}\label{Sec5}

In this section we will prove Theorem \ref{maintheorem3}. 
\begin{lem}\label{estiep1}
There exists a constant $C > 0$ such that 
\begin{align*}
&\|\mu_{\ep}\|_{L^2(0, T; W)} 
+ \|\varphi_{\ep}\|_{L^{\infty}(0, T; V_{\ep})} 
+ \|\varphi_{\ep}\|_{H^1(0, T; H)} 
+ \|\theta_{\ep}\|_{L^2(0, T; V)} 
\\ 
&+ \|B_{\ep}(\varphi_{\ep})\|_{L^2(0, T; H)} 
+ \|\xi_{\ep}\|_{L^2(0, T; H)} 
+ \|(\ln\theta_{\ep})_{t}\|_{L^2(0, T; V^{*})}
+ \|\ln\theta_{\ep}\|_{L^{\infty}(0, T; H)} 
\leq C
\end{align*}
for all $\ep \in (0, 1)$. 
\end{lem}
\begin{proof}
We can obtain this lemma 
by Lemmas \ref{estieplam1}-\ref{estieplam5}. 
\end{proof}

\bigskip 

\begin{prth1.3}
Lemma \ref{estiep1} yields that 
there exist some functions $u$, $\theta$, $\mu$, $\varphi$, $\xi$, $\eta$ 
such that 
\begin{align*}
    &\theta \in L^2(0, T; V),\ 
      u \in H^1(0, T; V^{*}) \cap L^{\infty}(0, T; H), \\ 
    &\mu \in L^2(0, T; W), \\ 
    &\varphi \in H^1(0, T; H) \cap L^{\infty}(0, T; H), \\
    &\eta, \xi \in L^2(0, T; H)
    \end{align*} 
and 
\begin{align}
&\ln\theta_{\ep} \to u 
\quad \mbox{weakly$^*$ in}\ H^1(0, T; V^*) \cap L^{\infty}(0, T; H), 
\label{weakep1} 
\\[1.5mm] 
&\ln\theta_{\ep} \to u 
\quad \mbox{strongly in}\ C([0, T]; V^*), 
\label{strongep1} 
\\[1.5mm] 
&\theta_{\ep} \to \theta 
\quad \mbox{weakly in}\ L^2(0, T; V), 
\label{weakep2} 
\\[1.5mm]
&\varphi_{\ep} \to \varphi 
\quad \mbox{weakly$^*$ in}\ H^1(0, T; H) \cap L^{\infty}(0, T; H), 
\label{weakep3} 
\\[1.5mm]
&\varphi_{\ep} \to \varphi 
\quad \mbox{strongly in}\ C([0, T]; V^*), 
\label{strongep2}
\\[1.5mm]
&B_{\ep}(\varphi_{\ep}) \to \eta 
\quad \mbox{weakly in}\ L^2(0, T; H), 
\label{weakep4}
\\[1.5mm]
&\mu_{\ep} \to \mu 
\quad \mbox{weakly in}\ L^2(0, T; W), 
\label{weakep5}
\\[1.5mm]
&\xi_{\ep} \to \xi 
\quad \mbox{weakly in}\ L^2(0, T; H) 
\label{weakep6}
\end{align}
as $\ep = \ep_{j} \searrow 0$. 
We see from \eqref{strongep1} and \eqref{weakep2} that 
\begin{align*}
&\int_{0}^{T} (\ln \theta_{\ep}(s), \theta_{\ep}(s))_{H}\,ds 
\notag \\
&= \int_{0}^{T} \left\langle 
                    \ln \theta_{\ep}(s), \theta_{\ep}(s) 
                    \right\rangle_{V^{*}, V}\,ds 
\notag \\ 
&\to \int_{0}^{T} \langle u(s), \theta(s) \rangle_{V^{*}, V}\,ds 
= \int_{0}^{T} (u(s), \theta(s))_{H}\,ds 
\end{align*}
as $\ep = \ep_{j} \searrow 0$, 
and then it holds that 
\begin{equation}\label{pizza1}
u = \ln \theta  \quad \mbox{a.e.\ in}\ \Omega\times(0, T)
\end{equation}
(see, e.g., \cite[Lemma 1.3, p.\ 42]{Barbu1}). 
We can verify that 
\begin{equation}\label{pizza2}
\varphi_{\ep} \to \varphi 
\quad \mbox{strongly in}\ C([0, T]; H) 
\end{equation}
as $\ep = \ep_{j} \searrow 0$ 
by Lemmas \ref{PL3}, \ref{estiep1} and \eqref{strongep2}. 
It follows from \eqref{weakep6} and \eqref{pizza2} that 
\[
\int_{0}^{T} (\xi_{\ep}(s), \varphi_{\ep}(s))_{H}\,ds 
\to \int_{0}^{T} (\xi(s), \varphi(s))_{H}\,ds 
\] 
as $\ep = \ep_{j} \searrow 0$, 
whence we have 
\begin{equation}\label{pizza3}
\xi \in \beta(\varphi)  \quad \mbox{a.e.\ in}\ \Omega\times(0, T). 
\end{equation}
Since $\{\varphi_{\ep}\}_{\ep \in (0, 1)}$ is bounded in $L^{\infty}(0, T; V_{\ep})$ 
by Lemma \ref{estiep1}, 
we can obtain 
\begin{equation}\label{pizza4}
\varphi \in L^{\infty}(0, T; V)
\end{equation}
by the Ponce criterion \cite[Theorem 1.2]{P2004}. 
Thanks to Lemmas \ref{PL1} and \ref{PL2}, 
we can show  
\begin{align}\label{pizza5}
\eta = - \Delta \varphi \quad \mbox{a.e.\ in}\ \Omega\times(0, T),  
\qquad 
\partial_{\nu}\varphi = 0  \quad \mbox{a.e.\ on}\ \Gamma\times(0, T)
\end{align}
in reference to \cite[Section 5.1]{DST2021}. 
Therefore we can prove Theorem \ref{maintheorem3} 
by \eqref{weakep1}--\eqref{pizza5}, {\bf C3}, {\bf C7}. 
\qed
\end{prth1.3}

%==========================================================
%%%%%%%                                             %%%%%%%
  %%%                                                 %%%
 %%%                                                   %%%
%%%                    Acknowledgments                  %%%
 %%%                                                   %%%
  %%%                                                 %%%
%%%%%%%                                             %%%%%%%
%==========================================================
%\smallskip
%\section*{\blue{Acknowledgments}}

%\\

%
%
%
%%==============================================================%%
%%==============                                  ==============%%
%%======                                                  ======%%
%%====                                                      ====%%
%%==                         Reference                        ==%%
%%====                                                      ====%%
%%======                                                  ======%%
%%==============                                  ==============%%
%%==============================================================%%

\end{document}